\pdfoutput=1

\documentclass[12pt]{amsart}

\usepackage[centertags]{amsmath}
\usepackage{amsfonts}
\usepackage{amssymb}
\usepackage{amsthm}
\usepackage{ subfigure}
\usepackage{graphicx}
\usepackage{color}
\usepackage{enumerate}
\usepackage{diagrams}
\usepackage{fullpage}

\newtheorem{theorem}{Theorem}[section]
\newtheorem{proposition}[theorem]{Proposition}

\newtheorem{corollary}[theorem]{Corollary}

\theoremstyle{definition}
\newtheorem{definition}[theorem]{Definition}

\begin{document}

\newcommand{\V}{\boldsymbol{V}} 
\newcommand{\s}{\mathcal{S}}
\newcommand{\e}{\mathrm{e}^}

\title{The Tutte-Potts connection in the presence of an external magnetic field}

\author[J.~Ellis-Monaghan]{Joanna A. Ellis-Monaghan}
\address{Department of Mathematics, Saint Michael's College, 1 Winooski Park, Colchester, VT 05439, USA.  }
\email{jellis-monaghan@smcvt.edu}
\thanks{The work of the first author was supported by the Vermont EPSCoR.}

\author[I.~Moffatt]{ Iain Moffatt}
\address{Department of Mathematics and Statistics,  University of South Alabama, Mobile, AL 36688, USA}
\email{imoffatt@jaguar1.usouthal.edu}



\date{\today}

\maketitle

\begin{abstract}
The classical relationship between the Tutte polynomial of graph theory and
the Potts model of statistical mechanics has resulted in valuable interactions between the
disciplines. Unfortunately, it does not include the external magnetic
fields that appear in most Potts model applications. Here we define the $\V$-polynomial,
which lifts the classical relationship between the Tutte polynomial
and the zero field Potts model to encompass external magnetic fields. The
$\V$-polynomial generalizes Noble and Welsh's $W$-polynomial, which extends the Tutte polynomial
by incorporating vertex weights and adapting contraction to accommodate them.
We prove that the variable field Potts model partition function (with its many specializations)
is an evaluation of the $\V$-polynomial, and hence a polynomial with deletion-contraction
reduction and Fortuin-Kasteleyn type representation. This unifies an important segment of Potts model theory and brings previously successful combinatorial machinery, including
complexity results, to bear on a wider range of statistical mechanics models.

This e-print is an extended version, including additional background information and more detailed proofs, of a paper of the same name that is to appear in Advances in Applied Mathematics.
\end{abstract}


\section{Introduction}

The classical relationship between the Tutte polynomial of graph theory and the Potts model of statistical mechanics has resulted in valuable interactions between the fields, particularly in the investigation of zeros and computational complexity.  Unfortunately, the classical theory does not encompass the external magnetic field that appears in most Potts model applications. Our main result is a new graph polynomial, $\V$, that lifts the well known relation between the classical Tutte polynomial and the zero field Potts model to the full Potts model with external magnetic field, and provides a single polynomial for the study of previously disparate Potts models.  It furthermore specializes to the previously known result relating the classical Tutte polynomial and the zero field Potts model.  This provides a framework for extending results from Tutte polynomial investigations to the variable field Potts model, and vice versa. In particular, the relationship shows that the variable field partition function is a polynomial in a set of expressions, gives  a deletion-contraction reduction that is conducive to induction arguments, and provides a Fortuin-Kasteleyn-type representation for the model. Existing results for the  $W$- and $U$-polynomials then make connections between the variable field Potts model and knot theory, and also provide computational complexity results for variable field Potts models. 

The Potts model of statistical mechanics models how micro-scale nearest
neighbour energy interactions in a complex system determine the macro-scale
behaviours of the system.  This model plays an important role in the theory of
phase transitions and critical phenomena in physics, and has applications as
widely varied as magnetism, tumor migration, foam
behaviours, and social demographics. The underlying network of the system is typically modelled by a graph, with the vertices representing molecules, cells, bubbles, households, \emph{etc}., with edges between `neighbouring' units, \emph{i.e.} those that may influence each other.   The vertices may have different properties called \emph{spins} assigned to them, representing anything from magnetic spin to socio-economic status of the household, and neighbouring vertices have an interaction energy.  A \emph{state} of a graph is a choice of spin at each vertex.  The Potts model theory has its origins in the study of magnetism, and we use that terminology, for example spins and states, here. Two recent surveys of the Potts model and its connections with graph theory are \cite{BE-MPS10} and \cite{WM00}.

The $q$-state Potts model partition function is the normalization factor for the
Boltzmann probability distribution and is given by: $Z(G) = \sum {\e{- \beta (h(\sigma ))}}$ where the sum is over all possible states $\sigma$ of a graph $G$ using $q$ spins. Here $\beta$ is a function of the temperature, and the Hamiltonian $h$ is a measure of the energy of the state.  The Hamiltonian may have many different forms depending on the specific application.  In the absence of an external magnetic field and with constant interaction energies between adjacent vertices, the Hamiltonian of a state $\sigma$ is simply  $h(\sigma ) =  - J\sum\limits_{\{i,j\}  \in E(G)} {\delta(\sigma _i ,\sigma _j)}$, where $\delta$ is the Kronecker $\delta$-function. In this case, the Potts model partition function is equivalent to the classical Tutte polynomial:  
\[Z\left( {G;q,\beta } \right) = {q^{k\left( G \right)}}{v^{\left| {v\left( G \right)} \right| - k\left( G \right)}}T\left( {G;\frac{{q + v}}
{v},v + 1} \right),\]
where $v = \e{J\beta } - 1$.  (See \cite{FK72} for the nascent stages of this theory and later exposition in \cite{Bax82, BE-MPS10, Big93, Bol98, Loe10, Tut84,  Wel93, WM00}).  This relationship has resulted in a remarkable synergy between the areas of combinatorics and statistical mechanics, particularly for computational complexity and the study of the zeros of these polynomials. (See \cite{BE-MPS10, GJ07, Roy09, Sok05, WM00}  for surveys of these results.) 

The equivalence of the Tutte polynomial and the Potts model partition function assumes the absence of an external field.  However, many applications of the Potts model depend on additional terms in the Hamiltonian corresponding to the presence of such additional influences, for example the standard models of magnetism, the cellular Potts model of \cite{GG92}, and also see \cite{O+03} for examples in the life sciences.  Furthermore, within many of these models, there is a need for edge dependent interaction energies and site dependent external fields.  Significant work has been done in incorporating the edge dependent interaction energies, both from a combinatorial perspective with edge weighted generalizations of the Tutte polynomial (see \cite{BR99, E-MT06, Zas92}) and from a statistical mechanics approach (see \cite{Sok05} for a survey).  Only very recently though have external fields been investigated in the context of graph polynomials.  In the special case of the Ising model, essentially the Potts model with $q=2$, an interesting new polynomial has been found in \cite{WF}.  This polynomial has a deletion-contraction relation and captures approximating functions, but does not specialize to the Tutte polynomial, nor does the deletion-contraction extend to non-constant magnetic fields. Also, some weighting strategies have been developed to study non-zero magnetic fields in a graph colouring framework (see \cite{CS09, CS10, SX10}), with an accompanying graph polynomial.  These papers show that the Potts model partition function with an external field term does not have a traditional deletion-contraction reduction, and in fact give details for its deviation from this.

Here however, by using a definition of contraction that incorporates vertex weights, we are able to assimilate a Hamiltonian of the following generic form into the theory of the Tutte-Potts connection:
\begin{equation}\label{e.ham}  h(\sigma) =  -\sum_{ \{ i,j \} \in E(G) }  J_{i,j} \delta(\sigma_i, \sigma_j)  - \sum_{v_i\in V(G)}    \sum_{\alpha =1}^q M_{i,\alpha}  \delta(\alpha,  \sigma_i)  ,\end{equation}
where a magnetic field vector 
 \[ \boldsymbol{M}_i :=\{  M_{i,1}, M_{i,2}, \ldots , M_{i,q}  \}  \]
is associated to each vertex $v_i$. From this generic form we are able to specialize to various forms of the Hamiltonian with external fields that are common in the physics literature.  

Our approach is motivated by the $W$- and $U$-polynomials of Noble and Welsh \cite{NW99}.  These graph polynomials were originally developed in the context of knot theory through an investigation of the combinatorics  behind Chmutov, Duzhin and Lando's work on Vassiliev invariants (see \cite{CDL94I, CDL94II, CDL94III}). They generalize the classical Tutte polynomial by incorporating vertex weights, while retaining essential properties such a deletion-contraction reduction, state sum formulation, and universality (or recipe theorem). Like the classical Tutte polynomial, the $W$-polynomial has a deletion-contraction reduction, but it begins with positive integer weights on the vertices, and these weights are summed if an edge joining them is contracted. We define a $\V$-polynomial that extends the $W$-polynomial by incorporating edge weights to encode variable interaction energies and allowing vertex weights in a semigroup, in particular  the vector space $\mathbb{C}^q$.

If we think of the weighted spins in the Hamiltonian as vertex weights, then the additional term  $\sum_{i\in V(G)}    \sum_{\alpha =1}^q M_{i,\alpha}  \delta(\alpha,  \sigma_i) $
is handled in the deletion-contraction relations by summing the vector valued vertex weights exactly as done by the $U$-, $\V$-, and $W$-polynomials. We use this insight about the vertex weights to prove that the Potts model partition function with an external field is an evaluation of the $\V$-polynomial. This gives the desired deletion-contraction reduction for the external field Potts model, facilitating induction arguments.  More importantly, these partition functions may now be expressed as polynomials and Fortuin-Kasteleyn-type representations, which facilitates Taylor expansions and computer simulations used to estimate critical exponents and phase transitions. 

Since the classical Tutte polynomial is a specialization of the $\V$-polynomial, this relationship contains the original relationship between the classical Tutte polynomial and zero field Potts model.  This new relationship for the variable field Potts model extends the computational and analytic tools available to statistical mechanics applications.  In particular, we are able to immediately transfer computational complexity results for the $W$- and $U$- polynomials to the extended Potts model partition functions. Other results, such as those on zeros and phase transitions in the classical and multivariate settings, might also be adapted to this new context. Furthermore, provocative new connections between knot theory and statistical mechanics arise from this theory.

\section{Background}

\subsection{The $q$-state Potts model}

We begin by recalling some basic information about the Potts model.  
Let $G$ be a graph and consider a set of $q$ elements, called \emph{spins}.  In the
abstract, the spins may be numbers or colours, but typically they are values
relevant to some specific application. For example, in studying uniaxial
magnetic materials, $q = 2$, and the possible spins are $ + 1$ and $ - 1$.  
A \emph{state} of a graph $G$ is an assignment of a single spin to each vertex
of the graph. Combinatorially, a graph state may be thought of simply as a colouring (not necessarily proper) of the graph. Thus, if the vertex set of $G$ is $V(G)=\{v_1, \ldots, v_n\}$, then a {\em state} of $G$ is an assignment  $\sigma: V(G)\rightarrow \{1,\ldots , q\} $, for $q\in \mathbb{Z}^+$. We let $\s(G)$ denote the set of states of $G$. 

For greater readability we let $\sigma_i := \sigma (v_i)$, for $\sigma \in \s(G)$. In addition, to simplify some of the summation formulas, at times we will use the indices $i=1,\ldots, n$ of the vertices in place of the vertices, for example we may denote an edge  $e=\{v_i, v_j\}$  by $e=\{i,j\}$, and a vertex $v_i$ as $i$. This convention will always be clear from context and should cause no confusion.  

The interaction energy may be thought of simply as a weight on edge of the graph. In physics applications, the interaction energies are typically real numbers, with the model called \emph{ferromagnetic} if all are positive, and \emph{antiferromagnetic} if all are negative.  However, in our context of graph polynomials, they may simply be taken to be independent commuting variables. We will denote the interaction energy on an edge $e=\{v_i, v_j\}$ by $J_e:=J_{i,j} (= J_{v_i,v_j})$, and let
\[       j(\sigma) :=  -\sum_{ \{ i,j \} \in E (G)} J_{i,j}  \delta(\sigma_i, \sigma_j)  =
 -\sum_{ e \in E (G)} J_{e}   \delta(\sigma_i, \sigma_j) . \]

The \emph{Hamiltonian} is a measure of the energy of a state.  We
begin with the simplest formulation, where the interaction energy is a constant $J$, and the
Hamiltonian depends only on the nearest neighbour interactions (without any
external field or other modifying forces).  

\begin{definition} \label{zero Hamiltonian}
The \emph{zero field Hamiltonian} is

\begin{equation} \label{z.ham}
h(\sigma ) =  - J\sum\limits_{ij  \in E(G)} {\delta(\sigma _i ,\sigma _j)},
\end{equation}
where $\sigma$ is a state of a graph $G$, where $\sigma_i $ is the spin at vertex
$i$, and where $\delta $ is the Kronecker delta function.

\end{definition}

However, to encompass Hamiltonians with various other external fields and variable interaction energies, we need a much more general form.

\begin{definition} \label{general Hamiltonian}

Let $G$ be a graph.  Assign to each edge $e$ an interaction energy of $J_e $, and assign to each vertex $v_i$ a magnetic weight vector  \[ \boldsymbol{M}_i :=(  M_{i,1}, M_{i,2}, \ldots , M_{i,q}  ). \]
 
Then the  \emph{Hamiltonian of  the Potts model with variable edge interaction energy and variable magnetic field} is 
\begin{equation}\label{e.ham2}  
h(\sigma) =  -\sum_{ \{ i,j \} \in E }  J_{i,j} \delta(\sigma_i, \sigma_j)  - \sum_{v_i\in V(G)}    \sum_{\alpha =1}^q M_{i,\alpha}  \delta(\alpha,  \sigma_i).
\end{equation}

\end{definition}

Note that the Hamiltonian of Equation \ref{z.ham} may be recovered from the general Hamiltonian of Equation \ref{e.ham2} by setting $J_e =J$ for all edges $e$, and taking all the magnetic field vectors to be zero.  We will recover various other Hamiltonians of relevance to physics applications from the general Hamiltonian of Equation \ref{e.ham2} in a similar way.

Regardless of the choice of Hamiltonian, the Potts model partition function is the sum over all possible states of an exponential function of the Hamiltonian, as follows:

\begin{definition}  Given a set of
$q$ spins and a Hamiltonian $h$, then the \emph{$q$-state Potts
model partition function} of a graph $G$ is
\begin{equation*}
Z(G) = \sum_{\sigma \in \s (G)}   \e{-\beta h(\sigma)},
\label{Z}
\end{equation*}
where $\beta = 1/(\kappa T)$, where $T$ is the temperature of the system, and where $\kappa
= 1.38 \times 10^{-23} $ joules/Kelvin is the Boltzmann constant.

\end{definition}

The Potts model partition function is the normalization factor for the
Boltzmann probability distribution.  That is, the probability of a system in thermal equilibrium with its environment being in a particular state
$\sigma $ at temperature $T$ is: 

\begin{equation} \label{Boltzmann prob}
Pr\left( {\sigma ,\beta } \right) = \exp ( -
\beta h (\sigma ))/Z(G).
\end{equation}

Important thermodynamic functions such as internal energy,
specific heat, entropy, and free energy may all be expressed in terms of the partition function.  An important goal of statistical mechanics is to determine phase transition temperatures,
that is, critical temperatures around which a small change in temperature
results in an abrupt, non-analytic change in various physical
properties.  Since the thermodynamic functions can be given in terms of the partition function, determining the analyticity of the partition function in the infinite volume limit (over a class of well-behaved graphs of increasing size), both theoretically and through computer simulations, is central to this theory.  This has led to considerable study of the zeros of the partition function and its computational complexity.  

\subsection{The $U$- and $W$-polynomials}\label{s.w}

Since our methods build on Noble and Welsh's paper \cite{NW99}, we recall some definitions and results from their work.

As usual, $E(G)$ and $V(G)$ are the edge set and vertex set, respectively, of a graph $G$.  If $A \subseteq E(G)$, then $n(A)$, $r(A)$, and $k(A)$ are, respectively, the nullity, rank, and number of components of the spanning subgraph of $G$ with edges in $A$.  A {\em vertex weighted graph} consists of a graph $G$, with vertex set $V(G)=\{v_1, v_2, \ldots , v_n\}$ equipped with a  weight function $\omega :  V(G)\rightarrow \mathbb{Z}^+$. The {\em weight} of the vertex $v_i$ is the value $\omega (v_i)$. For simplicity of exposition we will often write $\omega_i$ rather than $\omega (v_i)$, and $\boldsymbol{x}$ for an indexed set of variables, \emph{e.g.} $\boldsymbol{x}=\{x_k| k \in S\}$ for some set $S$. 

\begin{definition}\label{dc}
If $G$ is a vertex weighted graph with weight function $\omega$, and $e$ is an edge of $G$, then:
\begin{enumerate}
\item  If $e$ is any edge of $G$, then $G-e$ is the graph obtained from $G$ by deleting the edge $e$ and leaving the weight function unchanged;

\item If $e$ is any non-loop edge of $G$, then $G/e$ is the graph obtained from $G$ by contracting the edge $e$ and changing  the weight function as follows: if $v_i$ and $v_j$ are the vertices incident to $e$, and $v$ is the vertex of $G/e$ created by the contraction, then $ \omega(v) =  \omega(v_i) +\omega(v_j) $.  Loops are not contracted.
\end{enumerate}
\end{definition}
The deletion and contraction of an edge in a vertex weighted graph is illustrated in Figure~\ref{f.dc}. 

\begin{figure}
\[\begin{array}{ccccc}
\includegraphics[width=3cm]{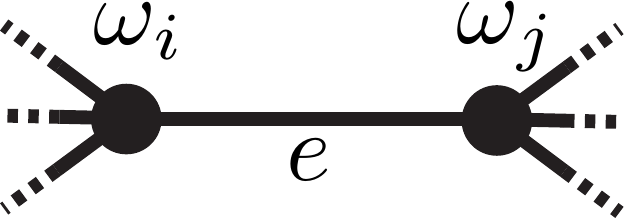} & \hspace{1cm}&\includegraphics[width=3cm]{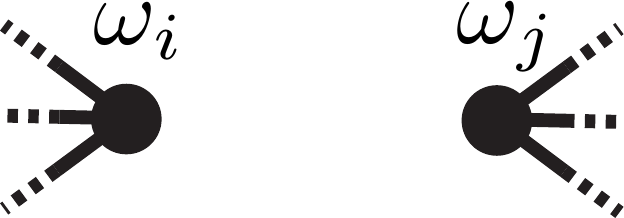} &\hspace{1cm}&\includegraphics[width=3cm]{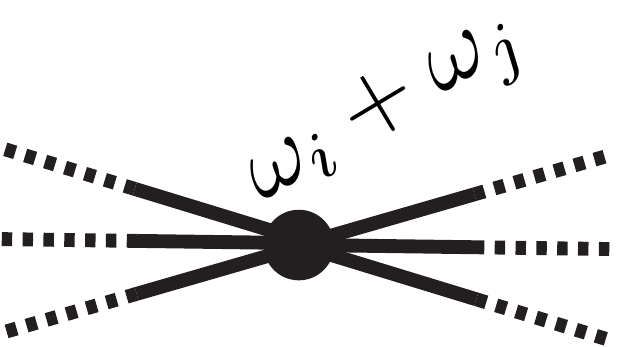}  \\
G && G-e && G/e
\end{array}
\]
\caption{Deletion and contraction of an edge  in a vertex weighted graph.}
\label{f.dc}
\end{figure}

We now recall Noble and Welsh's $W$-polynomial.

\begin{definition}[Noble and Welsh \cite{NW99}]\label{def.Wpoly}
 Let $G$ be a vertex weighted graph and $\{y, x_1, x_2, \ldots\}$ be a  set of commuting indeterminates.  
The {\em $W$-polynomial} of $G$,
\[ W(G)=W(G; \boldsymbol{x},y)\in \mathbb{Z} [  \{x_i\}_{i\in \mathbb{Z}^+}, y ]\]
is defined recursively by
\begin{enumerate}
\item if $e$ is not a loop then 
\[  W(G)= W(G-e)+W(G/e)  ,\]
where deletion and contraction act as in Definition~\ref{dc};
\item if $e$ is a loop, then \[W(G)=yW(G-e);\]
\item if $E_m$ consists of $m$ isolated vertices of weights $\omega_1, \ldots , \omega_m$, then 
\[  W(E_m)= \prod_{i=1}^{m} x_{\omega_i}.\]
\end{enumerate}
\end{definition}

The $W$-polynomial satisfies a recipe theorem (or universality theorem) that is analogous to the recipe of the Tutte polynomial (\cite{OW79}).  

\begin{theorem}[Noble and Welsh \cite{NW99}]\label{t.nwuniv}
Let $f$ be a function on vertex weighted graphs defined recursively by the following conditions:
\begin{enumerate}
\item if $e$ is not a loop then 
\[  f(G)= af(G-e)+bf(G/e)  .\]
\item If $e$ is a loop, then \[f(G)=yf(G-e).\]
\item If $E_m$ consists of $m$ isolated vertices of weights $\omega_1, \ldots , \omega_m$, 
\[  f(E_m)= \prod_{i=1}^{m} x_{\omega_i}.\]
\end{enumerate}
Then, if $a$ and $b$ are non-zero,
\[ f(G)=a^{|E(G)|-|V(G)|}b^{|V(G)|}  W\left(  G; \frac{a\boldsymbol{x} }{b} , \frac{y}{a}\right).  \]
\end{theorem}

For reference later, we record the  following properties of the $W$-polynomial.
\begin{theorem}[Noble and Welsh \cite{NW99}]\label{t.nwprops} Let $G$ be a vertex weighted graph. Then
\begin{enumerate}
\item \label{t.nwprops1} $W(G)$ can be represented as a sum over spanning subgraphs:
\[ W(G)=\sum_{A\subseteq E(G)}   x_{c_1}   x_{c_2}\cdots   x_{c_{k(A)}}   (y-1)^{|A|-r(A)},   \]
where $c_l$ is the sum of the weights of all of the vertices of the $l$-th component of the spanning subgraph $(V(G), A)$. 

\item \label{t.nwprops2} If $x_i=\theta$, for each $i$, then the resulting polynomial $W(x_i=\theta, y)$ is given by the Tutte polynomial:
\[  W(x_i=\theta, y) = \theta^{k(G)}  T(G;1+\theta , y).\]
\end{enumerate}
\end{theorem}

The \emph{$U$-polynomial}, implicitly studied by Chmutov, Duzhin and Lando in \cite{CDL94III}, was formally defined by Noble and Welsh in \cite{NW99} as an adaptation of the $W$-polynomial to unweighted graphs. The $U$-polynomial of an unweighted  graph is is defined simply by assigning the weight $1$ to each vertex and then taking the $W$-polynomial of the resulting weighted graph.

\section{The $\V$-polynomial for edge weighted graphs}

In this section we define the $\V$-polynomial, which extends the $W$-polynomial to graphs that have both edge weights and more general vertex weights.   An {\em edge weighted} graph $G$ is a graph equipped with a mapping $\gamma$ from its edge set  $E(G)$, to a set $\boldsymbol{\gamma}:=\{\gamma_e\}_{e\in E(G)}$.  Throughout this paper we will use the standard convention  $\gamma:e\mapsto \gamma_e$, for $e\in E(G)$. 
We incorporate the edges weights using what has become a `standard' technique (see for example \cite{ BR99, E-MT06, Mof08, Sok05, Zas92}) for extending Tutte-like polynomials to weighted graphs: we associate an indeterminate to each edge that then becomes a factor in the deletion-contraction reduction. 

We also generalize the vertex weights.  In the original definition of the $W$-polynomial (\cite{NW99}) the vertex weights are taken to be in $\mathbb{Z}^+$. Here however, we only require that the weights be in a torsion-free commutative semigroup (we have not explored the possibility of torsion, but this would be an interesting investigation). All that is actually necessary is that the vertex weights be additive and be a subset of the indexing set  $S$ of the variables $\boldsymbol{x}=\{x_k\}_{k\in S}$ of the polynomial.  In fact, the weights need not range over the whole indexing set, provided that every possible sum of the vertex weights in each component of $G$ appears as an index in $\boldsymbol{x}$. For the physics applications later in the paper, we will usually consider vertex weights in a $q$-dimensional complex vector space. In the definition of $\V$, as with $W$, when  a non-loop edge is contracted, the vertex weights on its ends are added.

The letter $V$ is logical in the progression $U$-, $V$-, $W$-, and we use boldface simply to distinguish the polynomial $\V(G)$ from the vertex set $V(G)$.

\begin{definition}\label{mvw}
Let $S$ be a torsion-free commutative semigroup, let $G$ be a graph equipped with vertex weights $\boldsymbol{\omega}:=\{ \omega_i \} \subseteq S$  and edge weights $\boldsymbol{\gamma}:= \{\gamma_e\}$, and let $\boldsymbol{x}=\{x_k\}_{k\in S}$ be a set of commuting variables. Then the {\em $\V$-polynomial} of the vertex and edge weighted graph $G$, 
\[\V(G) = \V(G, \omega; \boldsymbol{x}, \boldsymbol{\gamma}) \in \mathbb{Z} [ \{\gamma_e\}_{e\in E(G)}, \{x_k\}_{k\in S}  ], \] is defined recursively by: 
\begin{enumerate}
\item if $e$ is a non-loop edge of $G$,
\[  \V(G) = \V(G-e)+\gamma_e  \V(G/e), \]
where  contraction acts as described in Definition~\ref{dc};
\item if $e$ is a loop, then 
\[ \V(G)=(\gamma_e+1)\V(G-e);  \]
\item if $E_m$ consists of $m$ isolated vertices of weights $\omega_1, \ldots , \omega_m$, then 
\[  \V(E_m)= \prod_{i=1}^{m} x_{\omega_i}.\]
\end{enumerate}
\end{definition}

We now prove several basic results about the $\V$-polynomial. The proofs of most of these results are adaptations of Noble and Welsh's proofs of the corresponding results about the $W$-polynomial from \cite{NW99}.

\begin{proposition}\label{p.vwelldefined}
The polynomial $\V$ is well defined in the the sense that the polynomial is independent of the order in which the deletion-contraction relation is applied to the edges.
\end{proposition}
\begin{proof}
We  proceed by induction on the number of edges of $G$. The result clearly holds for all graphs with fewer than two edges. Assume that $G$ has at least two edges, and that $e$ and $f$ are distinct non-loop edges of $G$. 
First we calculate $\V(G)$ by applying the deletion-contraction relation to the edge $e$ first, and $f$ second:
\begin{multline}\label{e.p1}
\V(G)= \V(G-e)+\gamma_e \V(G/e) \\ =
\V([G-e]-f)+  \gamma_f \V([G-e]/f) +   \gamma_e \V([G/e]-f) +\gamma_e\gamma_f \V( [G/e]/f).
 \end{multline}
Applying the deletion-contraction relation to the edge $f$ first, and $e$ second gives
\begin{multline}\label{e.p2}
\V(G)= \V(G-f)+\gamma_f \V(G/f) \\ =
\V([G-f]-e)+  \gamma_e \V([G-f]/e) +   \gamma_f \V([G/f]-e) +\gamma_f\gamma_e \V( [G/f]/e)
 \end{multline}
Then, since $[G-e]-f\cong [G-f]-e$, $[G-e]/f\cong [G/f]-e$, $[G/e]-f \cong [G-f]/e$ and $[G/e]/f\cong [G/f]/e$, the right hand sides of (\ref{e.p1}) and (\ref{e.p2}) are equal. A similar argument shows the independence of order when $e$ or $f$ is a loop.   The result then follows by induction.

\end{proof}

\begin{theorem}\label{t.vstatesum}
$\V(G)$ can be represented as a sum over spanning subgraphs:
\[ \V(G)=\sum_{A\subseteq E(G)}   x_{c_1}   x_{c_2}\cdots   x_{c_{k(A)}}   \prod_{e\in A} \gamma_e ,   \]
where $c_l$ is the sum of the weights of all of the vertices in the $l$-th connected component of the spanning subgraph $(V(G), A)$. 
\end{theorem}
\begin{proof}
Let
\[ S(G):= \sum_{A\subseteq E(G)}   x_{c_1}   x_{c_2}\cdots   x_{c_{k(A)}}   \prod_{e\in A} \gamma_e.\]
We will show by induction on the number of edges of $G$ that $S(G)=\V(G)$.
If $G$ has no edges then the result is clearly true. So suppose that $e$ is a non-loop edge of $G$.  Then  

\begin{multline}
S(G)=\sum_{A\subseteq E(G)}   x_{c_1}   x_{c_2}\cdots   x_{c_{k(A)}}   \prod_{e\in A} \gamma_e \\=
 \left( \sum_{\stackrel{A\subseteq E(G)}{e\notin A} }  x_{c_1}   x_{c_2}\cdots   x_{c_{k(A)}}   \prod_{e\in A} \gamma_e\right)
+\left(\sum_{\stackrel{A\subseteq E(G)}{e\in A} }  x_{c_1}   x_{c_2}\cdots   x_{c_{k(A)}}   \prod_{e\in A} \gamma_e\right) .
\end{multline}
There is a natural bijection between the spanning subgraphs of $G$ that contain the edge 
$e$ and the spanning subgraphs of $G/e$ (given by $A\mapsto A-\{e\}$). Using this correspondence and the obvious correspondence between the spanning subgraphs of $G-e$ and the spanning subgraphs of $G$ that do not contain $e$, we can write the above expression as
\[ \left(\sum_{A\subseteq E(G-e)}  x_{c_1}   x_{c_2}\cdots   x_{c_{k(A)}}   \prod_{e\in A} \gamma_e\right)
+\gamma_e \left(\sum_{A\subseteq E(G/e) }  x_{c_1}   x_{c_2}\cdots   x_{c_{k(A)}}   \prod_{e\in A} \gamma_e\right) .\]
This is equal to $S(G-e)+\gamma_e S(G/e)$. A similar argument shows that $S(G)=(\gamma_e+1)S(G-e)$ when $e$ is a loop.  It then follows by induction that $S(G)=\V(G)$.  
\end{proof}

We now give a recipe theorem for $\V$.

\begin{theorem} \label{vrecp}
Let $f$ be a function on vertex  and edge weighted graphs defined recursively by the following conditions (where each $\alpha_e\neq 0$):
\begin{enumerate}
\item for any non-loop edge $e$ 
\[  f(G)= \alpha_e f(G-e)+\beta_e f(G/e),\]
with $a_e \neq 0$ for all $e$;
\item for any loop $e$
\[f(G)=(\alpha_e + \beta_e)f(G-e); \]
\item If $E_m$ consists of $m$ isolated vertices of weights $\omega_1, \ldots , \omega_m$, 
\[  f(E_m)= \prod_{i=1}^{m} x_{\omega_i}.\]
\end{enumerate}
Then 
\[ f(G)=  \left( \prod_{e\in E(G)}  \alpha_e \right)\V (G, \omega;  \boldsymbol{x} , \{ \beta_e/\alpha_e \}_{e\in E(G)} )   \]
\end{theorem}
\begin{proof}
Let \[ \widetilde{f} := \frac{ 1}{(\prod_{e\in E(G)}  \alpha_e)} \;  f,\]
and let $\gamma_e:=\beta_e/\alpha_e$. Then 
\[    \widetilde{f} (G) =  \widetilde{f} (G-e)  + \gamma_e  \widetilde{f} (G/e), \]
when $e$ is a non-loop edge of $G$,
and 
\[   \widetilde{f} (G) = (1+\gamma_e)  \widetilde{f} (G-e), \]
when $e$ is a loop. So $ \widetilde{f} $ is equal to the $\V$-polynomial and the result follows.
\end{proof}

The {\em multivariate Tutte polynomial} (see \cite{Sok05} and also \cite{Tra89}),  which we will denote by $Z_T$ to distinguish it from the partition function $Z$, is an extension of the Tutte polynomial to edge weighted graphs, defined by \[ Z_T(G; \theta, \boldsymbol{\gamma} ) := \sum_{A\subseteq E(G)}  \theta^{k(A)} \prod_{e\in A} \gamma_e.   \]

In the following theorem we prove that the $\V$-polynomial generalizes the $W$-polynomial and the multivariate Tutte polynomial.  We first show that the $W$-polynomial can be recovered from the $\V$-polynomial by setting all the $\gamma_i$ variables in the $\V$-polynomial equal to $y-1$. Then, given that the $\V$-polynomial is an edge weighted version of the $W$-polynomial, and that the $W$-polynomial extends the classical Tutte polynomial, it is hardly surprising that the $\V$-polynomial also extends the multivariate Tutte polynomial.

\begin{theorem}\label{vprops} 
The $\V$-polynomial generalizes both the $W$-polynomial and the multivariate Tutte polynomial.
\begin{enumerate}
\item If $\omega: V(G) \rightarrow \mathbb{Z}^+$ and $\gamma_e=(y-1)$ for each $e\in E(G)$, then 
\[\V(G,\omega; \boldsymbol{x}, \gamma_e=(y-1))  =(y-1)^{|V(G)|}  W\left(  G, \omega; \frac{\boldsymbol{x} }{y-1} , y\right).\]

\item  If $x_i=\theta$, for each $i\in \mathbb{Z}^+$, then, independent of the weights $\omega$, 
\[\V(G, \omega; x_i=\theta , \boldsymbol{\gamma})  =Z_T(G; \theta , \boldsymbol{\gamma}).\]
\end{enumerate}
\end{theorem}

\begin{proof}
For the first item, we apply the recipe theorem for $\V$ (Theorem \ref{vrecp}) to 

\[f(G)=(y-1)^{|V(G)|}  W\left(  G; \frac{\boldsymbol{x} }{y-1} , y\right).\]   

By Definition \ref{def.Wpoly}, 
\begin{enumerate}
  \item $f(G)=f(G-e) + (y-1)f(G/e)$ for $e$ a non-loop edge;
  \item $f(G)=yf(G-e)$ if $e$ is a loop;
  \item $f(E_n)=\prod_{i=1}^{n}{x_{\omega_i}}$.
\end{enumerate}

Thus, $f$ satisfies the criteria of Theorem \ref{vrecp} with $\alpha_e=1$ and $\beta_e=y-1$, for all $e$, and the result follows.

\medskip

The second item follows easily from the state sum for the $\V$-polynomial: 
\[  \V(G; x_i=\theta , \boldsymbol{\gamma}) = \sum_{A\subseteq E(G)} \left( \prod_{i=1 }^{k(A)} \theta \right)
 \left( \prod_{e\in A} \gamma_e \right) =   \sum_{A\subseteq E(G)} \theta^{k(A)}  \prod_{e\in A} \gamma_e   =Z_T(G; \theta , \boldsymbol{\gamma}). \]
The second item also follows from the first item and the second item in Theorem~\ref{t.nwprops}.

\end{proof}

Although Theorem \ref{vprops}  tells us that the $\V$-polynomial assimilates the multivariate Tutte polynomial $Z_T$, the relation between the $\V$-polynomial and the fully parameterized coloured Tutte polynomial (see \cite{ BR99, E-MT06, Zas92}) remains to be explored.


\section{The $\V$-polynomial and the Potts model in an external field with edge dependent interaction energies.}

We now come to our reason for creating the $\V$-polynomial.  Here we show that the Potts model partition function with the Hamiltonian $h(\sigma)$ from Definition~\ref{general Hamiltonian}  is an evaluation of the $\V$-polynomial. This provides a recursive deletion-contraction definition for the full Potts partition function with variable edge interaction energy and variable magnetic field, and also shows that it is a polynomial.

\begin{theorem}\label{ZV} Let $G$ be a graph equipped with a magnetic field vector  $\boldsymbol{M}_i=(M_{i,1}, \ldots , M_{i,q})\in \mathbb{C}^q$ at each vertex $v_i$, and suppose
 \[ 
h(\sigma) = - \sum_{ \{ i,j \} \in E }  J_{i,j}  \delta(\sigma_i, \sigma_j)  - \sum_{v_i\in V(G)}    \sum_{\alpha =1}^q M_{i,\alpha}  \delta(\alpha,  \sigma_i).
 \]
Then 
\begin{enumerate}
\item \label{zv1} if $e=\{v_i,v_j\}$ is a non-loop edge of $G$,
\[Z(G)=Z(G-e) + (\e{\beta J_{i,j}}-1) Z(G/e) ;\]
\item \label{zv3} if $e=\{v_i,v_i\}$ is a loop  \[Z(G)=\e{\beta J_{i,i}}Z(G-e)  ;\]
\item \label{zv2} $Z(E_n) =\prod_{i=1}^n X_{\boldsymbol{M}_i}$, where $E_n$ consists of $n$ isolated vertices of vector valued weights $\boldsymbol{M}_1, \boldsymbol{M}_2, \ldots , \boldsymbol{M}_n$, and, for any weight $\boldsymbol{M}_i$, 
\[  X_{\boldsymbol{M}_i}=  \sum_{\alpha=1}^q \e{\beta M_{i, \alpha}}.\]
\end{enumerate}
\end{theorem}

\begin{proof} 
To prove Item~\ref{zv1} of the theorem, let $e=\{v_a,v_b\}$ be a non-loop edge of $G$. 
Then
\begin{multline*}
 Z(G) = 
\sum\limits_{\sigma \in \s (G) } \e{-\beta h(\sigma) }  
= \sum\limits_{\sigma \in \s (G) }  \e{-\beta \,j(\sigma)}  \left(  \prod_{v_i\in V(G)}    \e{  \sum_{\alpha}  \beta M_{i,\alpha}  \delta(\alpha,  \sigma_i) }\right)
\\
=    \sum\limits_{\sigma \in \s (G) }  \e{-\beta \,j(\sigma)}  \left(  \prod_{v_i\in V(G)}    \e{   \beta M_{i,\sigma_i} }\right)   .
\end{multline*}
The last equality holds since $M_{i, \alpha }\delta(\alpha,  \sigma_i) =  M_{i, \sigma_i}$ if $\alpha = \sigma_i$, and is  zero otherwise.

We can separate  the above sum by collecting together the states in which $\sigma_a \neq \sigma_b$, and in which $\sigma_a = \sigma_b$ to get
\[
\sum\limits_{\stackrel{\sigma \in \s (G)}{  \sigma_a \neq \sigma_b   }}  \e{-\beta \,j(\sigma)}  \left(  \prod_{v_i\in V(G)}    \e{   \beta M_{i,\sigma_i} } \right)    
+\sum\limits_{\stackrel{\sigma \in \s (G)}{  \sigma_a = \sigma_b   }}  \e{-\beta \,j(\sigma)}  \left(  \prod_{v_i\in V(G)}     \e{   \beta M_{i,\sigma_i} } \right).
\]
If $\sigma    $ is any state of $G$, then there is a unique state
$\sigma' $ of $G - e$ where each vertex has the same spin as it has in
$\sigma $.  Using this correspondence, the expression above then becomes
\[
 \sum\limits_{\stackrel{\sigma \in \s (G-e)}{\sigma_a\neq \sigma_b}}  \e{-\beta \,j(\sigma)}  \left(  \prod_{v_i\in V(G-e)}   \e{   \beta M_{i,\sigma_i} }  \right)    
+ \e{\beta J_e}
\sum\limits_{\stackrel{\sigma \in \s (G-e)}{  \sigma_a = \sigma_b   }}  \e{-\beta \,j(\sigma)}  \left(  \prod_{v_i\in V(G-e)}  \e{   \beta M_{i,\sigma_i} }   \right).
\]
The left hand sum in the above expression is nearly $Z(G-e)$, but we are missing the states of $G - e$ where $\sigma_a= \sigma_b$.  So we simply add and subtract these states, to obtain
\begin{equation}\label{eq1}
Z(G)= \sum\limits_{\sigma \in \s (G-e)}  \e{-\beta \,j(\sigma)}  \left(  \prod_{v_i\in V(G-e)}   \e{\beta M_{i, \sigma_i}}  \right)    
+ (\e{\beta J_e}-1)
\sum\limits_{\stackrel{\sigma \in \s (G-e)}{  \sigma_a = \sigma_b   }}  \e{-\beta \,j(\sigma)}  \left(  \prod_{v_i\in V(G-e)}   \e{\beta M_{i ,\sigma_i}}  \right).
\end{equation}
The left-hand side is precisely  $Z(G-e)$.

 We now turn our attention to expressing the right-hand sum as 
$Z(G/e)$.  If $\sigma  $ is any state of $G - e$ with $\sigma_a = \sigma_b   $, then there is a unique state $\sigma''$ of
$G/e$, where the vertex resulting from identifying $a$ and $b$ has the common
spin $\sigma_a = \sigma_b   $, and each remaining  vertex of $G/e$ has the
same spin as it does in $\sigma $.  We immediately have $j(\sigma)=j(\sigma'')$. Now for a state $\sigma$ with $\sigma_a = \sigma_b $, we have
\[ \prod_{v_i\in V(G-e)}   \e{\beta M_{i ,\sigma_i}}  = 
\left(\prod_{\stackrel{v_i\in V(G-e)}{i\neq a \text{ or } b}}   \e{\beta M_{i ,\sigma_i}} \right)  \e{\beta ( M_{a, \sigma_a} +M_{b,\sigma_b}  )  } 
=\left(\prod_{\stackrel{v_i\in V(G-e)}{i\neq a \text{ or } b}}   \e{\beta M_{i ,\sigma_i}} \right)  \e{\beta ( M_{a, \sigma_a} +M_{b,\sigma_a}  )  },
\]
%
where the last equality holds because $\sigma_a=\sigma_b$. Notice that $M_{a, \sigma_a} +M_{b,\sigma_a} $ is the $\sigma_a$-th term of the vector $\boldsymbol{M}_a  +\boldsymbol{M}_b$.
If $v_c$ is the vertex of $G/e$ resulting from identifying $v_a$ and $v_b$, then $v_c$ has weight  $\boldsymbol{M}_c=\boldsymbol{M}_a  +\boldsymbol{M}_b$. Also, in the state $\sigma'$ of $G/e$ induced by the state $\sigma$ of $G$ as described above, the spin of $v_c$ is equal to the spins of $v_a$ and $v_b$.  Using the fact that $\sigma_c=\sigma_b=\sigma_a$, so that $M_{a,\sigma_a}+M_{b, \sigma_a}=M_{c, \sigma_c}$,  we can write the above product as
 \[ \prod_{v_i\in V(G/e)}   \e{\beta M_{i ,\sigma_i} }   = Z(G/e).\]
Equation~\ref{eq1} then gives 
\[  Z(G)=Z(G-e)+ (\e{\beta J_e}-1) Z(G/e)\] 
as required.

\medskip

Now, to prove Item~\ref{zv3} of the theorem,  suppose that $e=\{v_a,v_a\}$ is a loop. 
Then 
\[  Z(G)=
 \sum\limits_{\sigma \in \s (G) }  \e{\beta  \sum_{  \{i,j\} \in E(G)} J_{i,j}\, \delta(\sigma_i, \sigma_j)} \left(  \prod_{v_i\in V(G)}    \e{   \beta M_{i,\sigma_i} }\right)   \]
Since $e$ is a loop, both its ends are assigned the same spin in each state, so we can write the partition function as
\[   \e{\beta J_{a,a}}  \sum_{\sigma \in \s(G-e)}    \e{\beta  \sum_{\{i,j\} \in E(G-e)} J_{i,j}\,\delta_{i,j} }  \left(  \prod_{v_i\in V(G-e)}   \e{\beta M_{i, \sigma_i}}  \right) ,\]
where we have identified the states of $G$ and $G-e$ in the way described in the proof of Item~\ref{zv1} above. Thus, $Z(G)=\e{\beta J_{a,a}}  Z(G-e)$ as required.

\medskip

Finally, if $G$ consists of a single isolated vertex, it is easily checked that $Z(G)= \sum_{\alpha=1}^q \e{\beta M_{1, \alpha}}$. Item~(\ref{zv2}) then follows by the multiplicativity  of the partition function. 
\end{proof}

The following theorem is the main result of this paper. It states that the Potts model partition function with a variable external magnetic field and variable edge interaction is an evaluation of the $\V$-polynomial.

\begin{theorem}\label{ZV2} Let $G$ be a graph equipped with a magnetic field vector  $\boldsymbol{M}_i=(M_{i,1}, \ldots , M_{i,q})\in \mathbb{C}^q$ at each vertex $v_i$, and suppose
 \[ 
h(\sigma) = - \sum_{ \{ i,j \} \in E }  J_{i,j}  \delta(\sigma_i, \sigma_j)  - \sum_{v_i\in V(G)}    \sum_{\alpha =1}^q M_{i,\alpha}  \delta(\alpha,  \sigma_i).
 \]
Then 
\[  Z(G)=\V\left( G,\omega ; \;\{X_{\boldsymbol{M}}\}_{\boldsymbol{M}\in \mathbb{C}^q}    ,\;   \{  \e{\beta J_{i,j}}-1   \}_{\{i,j\}\in E(G)}      \right),\]
where the vertex weights  are given by $\omega (v_i) =\boldsymbol{M}_i$ 
 and, for any $\boldsymbol{M} = (M_1, \ldots , M_q) \in \mathbb{C}^q$, 
\[  X_{\boldsymbol{M}}=  \sum_{\alpha=1}^q \e{\beta M_{\alpha}}.\]
\end{theorem}
\begin{proof}
The equality of the $Z(G)$ and the $\V$-polynomial follows immediately from Theorem~\ref{ZV} and the recipe theorem for the $\V$-polynomial (Theorem~\ref{vrecp}). 
\end{proof}

As a corollary to this theorem we obtain a Fortuin-Kasteleyn-type representation for the Potts model  with variable external magnetic field and variable edge interaction.
 
\begin{corollary}\label{c.ZV1}
Let $G$ be a graph equipped with a magnetic field vector  $\boldsymbol{M}_i=(M_{i,1}, \ldots , M_{i,q})\in \mathbb{C}^q$ at each vertex $v_i$, 
 \[ 
h(\sigma) = - \sum_{ \{ i,j \} \in E }  J_{i,j}  \delta(\sigma_i, \sigma_j)  - \sum_{v_i\in V(G)}    \sum_{\alpha =1}^q M_{i,\alpha}  \delta(\alpha,  \sigma_i).
\]
Then 
\[   Z(G)= \sum_{A\subseteq E(G)}  X_{\boldsymbol{M}_{C_1}} \cdots X_{\boldsymbol{M}_{C_{k(A)}}} \left( \prod_{e\in A}  (\e{\beta J_e}-1) \right),\]
where $\boldsymbol{M}_{C_l}$ is the sum of the weights, $\boldsymbol{M}_i$, of all of the vertices $v_i$ in the $l$-th connected component of the spanning subgraph $(V(G), A)$, and  $X_{\boldsymbol{M}}=      \sum_{\alpha=1}^q \e{\beta M_{\alpha}}$.

\end{corollary}
\begin{proof}
The result follows immediately from Theorem~\ref{ZV2} and Theorem~\ref{t.vstatesum}. 
\end{proof}

Since the $\V$-polynomial of a graph $G$ is a polynomial in $\mathbb{Z} [  \{x_k\} ,\{\gamma_e\}] $, we can  use the relation between the $\V$-polynomial and the Potts model to express the  $q$-spin Potts model partition function  with variable edge interaction energy and variable magnetic field as a polynomial.

\begin{corollary} For a connected graph $G$ equipped with a magnetic field vector  $\boldsymbol{M}_i=(M_{1,1}, \ldots , M_{i,q}) \in \mathbb{C}^q$ at each vertex $v_i$, the $q$-spin Potts partition function  with variable edge interaction energy and variable magnetic field is a polynomial in the variables 

\[
\left\{\boldsymbol{v}, X_{\boldsymbol{M}} | \boldsymbol{M} \in  \mathcal{M} \right\},
\]
where $\boldsymbol{v} = \{ \e{\beta J_{e}}-1 \}_{e\in E(G)}$, and $X_{\boldsymbol{M}} = \sum_{\alpha=1}^q \e{\beta M_{\alpha}} $ for
\newline
 $\mathcal{M}:=\left.\left\{   \sum_{i=1}^{|V(G)|}  \varepsilon_i \boldsymbol{M}_i  \right|  \varepsilon_i = 0 \text{ or } 1  \right\} $.
\end{corollary}

This result may be extended to graphs with more than one component by using the multiplicativity of $\V$.



\section{Some specializations}\label{s.special}

As noted in the introduction, the $\V$-polynomial provides a unifying framework for Potts models with various Hamiltonians. Here we highlight some special forms of the magnetic weight vectors and resulting Hamiltonian that are frequently studied in the physics literature. We show how the associated partition functions are evaluations of the $\V$-polynomial and its specializations. This implies that these important models are also polynomials with deletion-contraction reductions and Fortuin-Kastelan-type representations. We also show how Theorem~\ref{ZV2} can be applied to the various Hamiltonians to give a hierarchy of relations among the resulting partition functions and the $W$-, $U$-, multivariate Tutte, or Tutte polynomials. These graph polynomials are all specializations of the $\V$-polynomial.

\subsection{A hierarchy of specializations}

The following result says that if we restrict to the Potts model with constant interaction energies and where all the magnetic field vectors are positive integer multiples of a given magnetic field vector, then the partition function is described by the $W$-polynomial.

\begin{theorem}\label{h1} Let $\boldsymbol{B}\in \mathbb{C}^q$, and let  $G$ be a graph equipped with a magnetic field vector  
$\boldsymbol{M}_i:=k_i\boldsymbol{B} = k_i (B_{1}, \ldots , B_{q}) $ at each vertex $v_i$, where $k_i\in \mathbb{Z}^+$. If the Hamiltonian is
\[ h(\sigma) = - J \sum_{ \{ i,j \} \in E(G) }    \delta(\sigma_i, \sigma_j)  - \sum_{v_i\in V(G)}    \sum_{\alpha =1}^q k_iB_{\alpha}  \delta(\alpha,  \sigma_i),
 \]
then 
\[  Z(G)=  (\e{\beta J}-1)^{|V(G)|}  W\left( G,\omega ; \;\{x_{k}\}_{k\in \mathbb{Z}^+}    ,\;    \e{\beta J}       \right),\]
where
\[  x_k
= \frac{1}{    \e{\beta J}-1  }\left(\sum_{\alpha=1}^q \e{\beta k B_{ \alpha}} \right),
\]
and $B_{\alpha}$ is the $\alpha$-entry of $\boldsymbol{B}$.
\end{theorem}

\begin{proof}
Setting each $J_{i,j}=J$ in Theorem~\ref{ZV2} gives 
\[  Z(G)=  \V \left( G,\omega ; \;\{X_{\boldsymbol{M}}\}_{\boldsymbol{M}\in \mathbb{Z^+}^q}    ,\;   \gamma_e= \e{\beta J}-1       \right),\]
where the $X_{\boldsymbol{M}}$ are defined in Theorem~\ref{ZV2}.
Each index $\boldsymbol{M}$, of $X_{\boldsymbol{M}}$ , is then of the form $k\boldsymbol{B}$ for some $k\in \mathbb{Z}^+$. Therefore $X_{\boldsymbol{M}} =X_{k\boldsymbol{B}}= \sum_{\alpha=1}^q \e{\beta k B_{ \alpha}} $, where $B_{\alpha}$ is the $\alpha$-entry of $\boldsymbol{B}$. 

Also since every index is of the form $k\boldsymbol{B}$, the indexing set is isomorphic to $\mathbb{Z}^+$.  Thus, we can choose our indexing set to be  $\mathbb{Z}^+$, letting $X_k
= \sum_{\alpha=1}^q \e{\beta k B_{ \alpha}}$.   
Applying Theorem~\ref{vprops} then gives 
\[  Z(G)=  (\e{\beta J}-1)^{|V(G)|} W\left( G,\omega ; \;\{x_k\}_{k\in \mathbb{Z}^+}    ,\;    \e{\beta J}       \right),\]
where $x_{k} = \left(\sum_{\alpha=1}^q \e{\beta k B_{ \alpha}}\right) /  \left(    \e{\beta J}-1  \right) $.
\end{proof}

If, in addition to constant edge interaction energy, all of the magnetic field vectors are fixed, then the Potts model partition function can be recovered from the $U$-polynomial. 

\begin{theorem}\label{h3} Let  $G$ be a graph equipped with a fixed magnetic field vector  
$\boldsymbol{B} = (B_1, \ldots , B_{\alpha})$ at each vertex, and Hamiltonian
\[ h(\sigma) = - J \sum_{ \{ i,j \} \in E(G) }    \delta(\sigma_i, \sigma_j)  - \sum_{v_i\in V(G)}    \sum_{\alpha =1}^q B_{\alpha}  \delta(\alpha,  \sigma_i),
 \]
then 
\[  Z(G)=  (\e{\beta J}-1)^{|V(G)|}  U\left( G,\omega ; \;\{x_{k}\}_{k\in \mathbb{Z}^+}    ,\;    \e{\beta J}-1       \right),\]
where
\[  x_k
= \frac{1}{    \e{\beta J}-1  }\left(\sum_{\alpha=1}^q \e{\beta k B_{ \alpha}} \right),
\]
and $B_{\alpha}$ is the $\alpha$-entry of $\boldsymbol{B}$.

\end{theorem}

\begin{proof}
The result follows easily from Theorem~\ref{h1} upon noting that the vector $1\boldsymbol{B}$ is the vertex weight of each vertex of $G$, and recalling that initial weights of 1 on all vertices is the defining property of $U$.

\end{proof}

The well known classical results which relate the Potts model partition function and the Tutte or multivariate Tutte polynomial in the cases when $M=0$ may also be recovered from Theorem~\ref{ZV} and Theorem~\ref{vprops} as follows.

\begin{theorem} \label{classical Tutte} Let $G$ be a graph. 
\begin{enumerate}
\item If $h(\sigma) =  -\sum_{ \{ i,j \} \in E }  J_{i,j} \delta(\sigma_i, \sigma_j)  $, then
\[Z(G)=Z_T( G;  q, \{ \e{\beta {J_{i,j}}} -1 \}_{\{i,j\}\in E(G)}) ;\]

\item and if $h(\sigma) =  -J \sum_{ \{ i,j \} \in E }   \delta(\sigma_i, \sigma_j)  $, then
\[Z(G) =
q^{k(G)} v^{|V(G)|-k(G)} T(G;(q + v)/v,v + 1),\] where $v = e^{\beta J}-1$.

\end{enumerate}
\end{theorem}

\bigskip

From the preceding theorems, we have the following hierarchy of relations among the $\V$-, $W$-, $U$-, and Tutte polynomials, and the partition functions with Hamiltonians of various degrees of generality. This hierarchy holds in the case of the indiscriminately preferred spins given in Theorems \ref{h1} through \ref{classical Tutte}, and adapts to the single preferred spin models of Theorems \ref{ZV4} through \ref{ZV5},  presumably with other models following a similar pattern as well.

{\small   
\begin{diagram}
 && & & \V   & \lInto &  \begin{array}{l}  \\ Z \text{ with variable interaction  }\\\text{energies and variable field} \end{array}  \\ 
 &&&\ruInto(2,3)&  \uInto && \\
 &&&  &W &\lInto &   \begin{array}{l}  \\ Z \text{ with constant interaction }\\\text{energies and integer scaled field} \end{array}  \\\ 
   \begin{array}{l} \\ Z \text{ with variable interaction }\\\text{energies and zero field} \end{array}  &\rInto&Z_T&& \uInto &&\\
 &&&  \luInto(2,3) & U	&\lInto &   \begin{array}{l} \\ Z  \text{ with constant interaction}\\\text{energies and constant field} \end{array}  \\
 &&& &\uInto&&\\
&&& &T  & \lInto & \begin{array}{l} \\ Z\text{ with constant interaction }\\\text{energies and zero field} \end{array}   
  \end{diagram}
}

\subsection{Some more examples}
There many variations of the Hamiltonian, having far too many forms to address here.  Thus, we select the following frequently studied forms of the Hamiltonian simply to illustrate how the theory presented here may be adapted to other settings.

One example of a Hamiltonian that is of particular interest in the physics literature models a system with an external field in which one particular spin (the first, without loss of generality) is preferred (see for example the surveys \cite{Ber05, Wu82}).  Here again, variable interaction energies and magnetic field values are allowed. In this case we have the following relation with the $\V$-polynomial. 
 
\begin{theorem}\label{ZV4}

Suppose a complex value $z_i$ is associated to each vertex $v_i$ of a graph $G$, and the Hamiltonian is given by
 \[ h(\sigma) = - \sum_{ \{ i,j \} \in E }  J_{i,j}  \delta(\sigma_i, \sigma_j)  - \sum_{v_i\in V(G)} z_i  \delta(1,  \sigma_i).
 \]
Then 
\[  Z(G)=\V\left( G,\omega ; \;\{X_z\}_{z\in \mathbb{C}}    ,\;   \{  \e{\beta J_e}-1   \}_{e \in E(G)}      \right),\]
where the vertex weights  are given by $\omega (v_i) =z_i$ 
 and 
\[  X_z=   \e{\beta z}+q-1.\]
\end{theorem}

\begin{proof}

We associate a magnetic field vector,  $\boldsymbol{M}_i= (z_i, 0, \ldots , 0)$, with each vertex  $v_i$ of $G$.  If we set 
\[\hat{h}(\sigma):=  -  \sum_{ \{ i,j \} \in E(G) } J_{i,j}   \delta(\sigma_i, \sigma_j)  - \sum_{v_i\in V(G)}    \sum_{\alpha =1}^q M_{i,\alpha}  \delta(\alpha,  \sigma_i),\]
then 
\[  \quad Z(G) = \sum_{\sigma \in \s (G)}   \e {-\beta h(\sigma)} =   \sum_{\sigma \in \s (G)}   \e {-\beta \hat{h}(\sigma)}. \]
Applying Theorem~\ref{ZV2} with the Hamiltonian $\hat{h}$ yields that

\[  Z(G)=  \V\left( G,\omega ; \;\{X_{\boldsymbol{M}}\}_{\boldsymbol{M}\in \mathbb{C}^q}    ,\;    \{\e{\beta J_e}-1\}_{e\in E(G)}       \right),\]
where the vertex weights are given by $\omega (v_i) =\boldsymbol{M}_i$ and, for any $\boldsymbol{M}\in \mathbb{C}^q$, 
\[  X_{\boldsymbol{M}}
= \sum_{\alpha=1}^q \e{\beta M_{ \alpha}},\]
where $M_{\alpha}$ is the $\alpha$-entry of $\boldsymbol{M}$.
However, the only indices $\boldsymbol{M}\in \mathbb{C}^q$ that actually appear are of the form $(z,0,\ldots,0)$ for some $z \in \mathbb{C}$, and if $\boldsymbol{M}$ has this form, then  $X_{\boldsymbol{M}}=\e{\beta z}+q-1$.  
Thus, we can choose our indexing set to be $\mathbb{C}$ instead of $\mathbb{C}^q$, so that 
\[  Z(G)=\V\left( G,\omega ; \;\{X_z\}_{z\in \mathbb{C}}    ,\;   \{  \e{\beta J_e}-1   \}_{e \in E(G)}      \right),\]
where $X_z=   \e{\beta z}+q-1$, as claimed.
\end{proof}

An analog of Theorem \ref{h3} and associated hierarchy may easily be derived: if all of the $z_i$'s are equal and $J_{i,j}=J$, then $Z(G)$ can be recovered from the $U$-polynomial; and if all of the $z_i$'s are of the form $k_iz$, for some fixed $z\in \mathbb{C}$ with the  $k_i$'s in $\mathbb{Z}^+$, and $J_{i,j}=J$,  then $Z(G)$ can be recovered from the $W$-polynomial.

When the Hamiltonian is that of Theorem \ref{ZV4}, a Fortuin-Kasteleyn-type representation for the Potts model  is well known. 
This Fortuin-Kasteleyn-type representation, given in Corollary~\ref{FK},  can be immediately recovered from Corollary~\ref{c.ZV1}.

\begin{corollary}\label{FK}

Suppose a complex value $z_i$ is associated to each vertex $v_i$ of a graph $G$, and the Hamiltonian is given by
 \[ h(\sigma) = - \sum_{ \{ i,j \} \in E }  J_{i,j}  \delta(\sigma_i, \sigma_j)  - \sum_{v_i\in V(G)} z_i  \delta(1,  \sigma_i).
 \]

Then 
\[   Z(G)= \sum_{A\subseteq E(G)}  X_{z_{C_1}} \cdots X_{z_{C_{k(A)}}} \left( \prod_{e\in A}  (\e{\beta J_e}-1) \right),\]
where $z_{C_l}$ is the sum of the weights, $z_i$, of all of the vertices $v_i$ in the $l$-th connected component of the spanning subgraph $(V(G), A)$, and  $X_z=     q-1 +  \e{\beta z}$. 
\end{corollary}

In \cite{Sok99}, Sokal studied the Potts model partition function in the case where the magnetic field vectors are of the form $\boldsymbol{M}_i=(M_{i,1}, M_{i,2}, \ldots, M_{i,r},  0, \ldots , 0)  $, where $0\leq r \leq q$  is fixed. With these magnetic field vectors, Theorem~\ref{ZV2} immediately gives the following result.

\begin{theorem}\label{ZV5}
Let $G$ be a graph where each vertex $v_i$ is equipped with a magnetic field vector  $\boldsymbol{M}_i=(M_{i,1}, M_{i,2}, \ldots, M_{i,r},  0, \ldots , 0)$, and suppose
 \[ h(\sigma) = - \sum_{ \{ i,j \} \in E (G)}  J_{i,j}  \delta(\sigma_i, \sigma_j)  - \sum_{v_i\in V(G)}    \sum_{\alpha =1}^q M_{i,\alpha}  \delta(\alpha,  \sigma_i).
 \]
Then 
\[  Z(G)=\V\left( G,\omega ; \;\{X_{\boldsymbol{M}}\}_{\boldsymbol{M}\in \mathbb{C}^q}    ,\;   \{  \e{\beta J_{a,b}}-1   \}_{\{a,b\}\in E(G)}      \right),\]
where the vertex weights  are given by $\omega (v_i) =\boldsymbol{M}_i$ 
 and, for any $\boldsymbol{M}\in \mathbb{C}^q$, 
\[  X_{\boldsymbol{M}}=   q-r + \sum_{\alpha=1}^r \e{\beta M_{\alpha}},\]
where $M_{\alpha}$ is the $\alpha$-entry of $\boldsymbol{M}$.
\end{theorem}

Again, an analog of Theorem \ref{h3} and associated hierarchy for this external field may easily be derived: if   each   $\boldsymbol{M}_i =k_i \boldsymbol{B}$, for $k_i \in \mathbb{Z}^+$, then $Z(G)$ can be recovered from the $W$-polynomial; and if all of the $\boldsymbol{M}_i $ are equal then $Z(G)$ can be recovered from the $U$-polynomial.

Also in \cite{Sok99}, Sokal found a Fortuin-Kasteleyn-type representation for the partition function used in the above theorem.
This Fortuin-Kasteleyn-type representation can be immediately recovered from Corollary~\ref{c.ZV1}: 
\begin{corollary}
Let $G$ be a graph where each vertex $v_i$ is equipped with a magnetic field vector  $\boldsymbol{M}_i=(M_{i,1}, M_{i,2}, \ldots, M_{i,r},  0, \ldots , 0)$. In addition let 
\[ h(\sigma) = - \sum_{ \{ i,j \} \in E(G) }  J_{i,j}  \delta(\sigma_i, \sigma_j)  - \sum_{v_i\in V(G)}    \sum_{\alpha =1}^q M_{i,\alpha}  \delta(\alpha,  \sigma_i) .
 \]
Then 
\[   Z(G)= \sum_{A\subseteq E(G)}  X_{\boldsymbol{M}_{C_1}} \cdots X_{\boldsymbol{M}_{C_{k(A)}}} \left( \prod_{e\in A}  (\e{\beta J_e}-1) \right),\]
where $\boldsymbol{M}_{C_l}$ is the sum of the weights, $\boldsymbol{M}_i$, of all of the vertices $v_i$ in the $l$-th connected component of the spanning subgraph $(V(G), A)$, and  $X_{\boldsymbol{M}}=     q-r + \sum_{\alpha=1}^r \e{\beta M_{\alpha}}$. 
\end{corollary}

Our final examples are variations of the Ising model (essentially the $q=2$ Potts model) used to study glassy behaviours (see \cite{HW05, NB99}).   The first is the Ising spin glass model, which has edge dependent random bond strengths $J_e$, but no external field. The second is the Random Field Ising Model (RFIM), used to study disordered states.  It has a random magnetic field in that the $z_i$'s are randomly chosen local magnetic fields that each affect only a single site.  To avoid redundancy of proof, we merge the two models in the following theorem (this generalization is also sometimes called the RFIM), and then recover each of them separately as corollaries.  We choose these examples in part because complexity results for Ising models (or $q = 2$ Potts models) differ significantly from those for $q > 2$ Potts models, both in the classical (no external field) case and the examples here. Furthermore, there have been two very recent studies of the Ising model from a graph theoretical perspective.  \cite{AM09} treats the Ising model with constant interaction energies and constant (but non-zero) magnetic field as a graph invariant and explores graph theoretical properties encoded by it.  \cite{WF} on the other hand, creates a new graph polynomial, which has a deletion-contraction reduction for non-loop edges, and which is, up to a prefactor and change of variables, equivalent to the RFIM.  We discuss this polynomial further in the conclusion. 

The Ising model takes spin values in $\{-1,1\}$. We will let  $\tau$  denote a state for the Ising model, which is a map $\tau:V(G)\rightarrow \{-1,+1\}$. As usual we set $\tau_i:=\tau(v_i)$. Also we will let $\mathcal{T}(G)$ be the set of states for the Ising model.

\begin{theorem}\label{Ising merge} Let $G$ be a graph with a vertex weight $\omega(v_i)=z_i\in \mathbb{C}$ associated to each vertex $v_i$, and suppose the Hamiltonian and partition function are given by
 \begin{equation}\label{e.Isingham}
  h(\tau) = - \sum_{ \{ i,j \} \in E(G) }  {J_{i,j} \tau_i \tau_j} - \sum_{v_i \in V(G)}{z_i \tau_i} \quad \text{and} \quad Z(G)=\sum_{\tau\in \mathcal{T}(G)} \e{-\beta h(\tau)}.
  \end{equation}
 Then
\[  Z(G)= \e{-\beta \left(\sum_{ e \in E }  J_e +3 \sum_{i \in V(G)} z_i  \right)}  \V\left( G,\omega  ; \;\{x_{z}\}_{z\in \mathbb{C}}    ,\;   \{  \e{2\beta J_{e}}-1   \}_{e\in E(G)}      \right),\]
where for any $z\in \mathbb{C}$, 
$  x_{z}=    \e{2z} + \e{4z} $.

Furthermore, $Z(G)$ is a polynomial in the variables $\left\{\boldsymbol{v},  x_{z} \; | \; z \in  \mathcal{M} \right\},$
where $\boldsymbol{v} = \{ \e{\beta J_{e}}-1 \}_{e\in E(G)}$, and $\mathcal{M}=\left.\left\{   \sum_{i=1}^{|V(G)|}  \varepsilon_i z_i  \right|  \varepsilon_i = 0 \text{ or } 1  \right\} $.

$Z(G)$ also has Fortuin-Kasteleyn-type representation
\[\sum_{A\subseteq E(G)}  x_{z_{C_1}} \cdots x_{z_{C_{k(A)}}} \left( \prod_{e\in A}  (\e{\beta J_e}-1) \right),\]
where $z_{C_l}$ is the sum of the weights, $z_i$, of all of the vertices $v_i$ in the $l$-th connected component of the spanning subgraph $(V(G), A)$. 
\end{theorem}

\begin{proof}
There is a natural bijection between $\mathcal{T}(G)$ and $\mathcal{S}(G) =\left\{\sigma: V(G)\rightarrow \{1,2\}  \right\}$, which, given $\tau\in \mathcal{T}(G)$,    is determined by setting $\sigma_i=1$ if $\tau_i=-1$, and $\sigma_i=2$ if $\tau_i=+1$. This determines a state $\sigma\in \mathcal{S}(G)$. 

Under this bijection  between $\mathcal{T}(G)$ and $\mathcal{S}(G)$, observe that   $\tau_i=2\sigma_i -3$ and $\tau_i\tau_j =2\delta (\sigma_i, \sigma_j) -1$, and so we may write the Hamiltonian in Equation~(\ref{e.Isingham}) as
\[
h(\sigma)= - 2\sum_{ \{ i,j \} \in E(G) } { J_{i,j} \delta (\sigma_i, \sigma_j)}- 2\sum_{V_i \in V(G)}{z_i\sigma_i} + \sum_{ \{ i,j \} \in E } { J_{i,j}} +3 \sum_{i \in V(G)}{z_i}  \]
With this Hamiltonian,  $Z(G)= \sum_{\tau \in \mathcal{T}(G)} {\e{-\beta h(\tau)}} =  \sum_{\sigma \in \mathcal{S}(G)} {\e{-\beta h(\sigma)}}$.

We now let 
\[
\tilde{h}(\sigma)= - \sum_{ \{ i,j \} \in E } { 2J_{i,j} \delta (\sigma_i, \sigma_j)}- \sum_{i \in V(G)}{2 z_i\sigma_i}.
\]
Then
\[  Z(G) = \e{-\beta\left(\sum_{ e \in E }  J_e +3 \sum_{i \in V(G)} z_i  \right)} \sum_{\sigma \in \mathcal{S}(G)} {\e{-\beta \tilde{h}(\sigma)}}.   \]
By applying Theorem~\ref{ZV2} we then have 
\[  Z(G)= \e{-\beta\left(\sum_{ e \in E }  J_e +3 \sum_{i \in V(G)} z_i  \right)}   \V\left( G,\omega ; \{X_{\boldsymbol{M}}\}_{\boldsymbol{M}\in \mathbb{C}^2} ,  \{ \e{2\beta J_e} -1   \}_{e\in E(G)}   \right),   \]
where the weight function $\omega$ is now $\omega_i= \boldsymbol{M}_i:=2z_i (1,2)\in \mathbb{C}^2$, and, if  $\boldsymbol{M}=z(1,2)$ then
$ X_{\boldsymbol{M}}  =   \sum_{\alpha=1}^2 \e{\beta M_{\alpha}} = \e{2z} + \e{4z} $.

Finally, observing that since each $\boldsymbol{M} = z(1,2)$, for some $z\in \mathbb{C}$, we can take $\mathbb{C}$ as the indexing set instead of $\mathbb{C}^2$. This gives, 
  \[  Z(G)= \e{-\beta\left(\sum_{ e \in E }  J_e +3 \sum_{i \in V(G)} z_i  \right)}   \V\left( G,\omega ; \{x_{z}\}_{z\in \mathbb{C}} ,  \{ \e{2\beta J_e} -1   \}_{e\in E(G)}   \right),   \]
where $x_z=   \e{2z} + \e{4z}   $. This gives the relation between the $\V$-polynomial and $Z(G)$.

The proofs that $Z(G)$ is a polynomial in the given expressions and has a Fortuin-Kasteleyn-type representation are also very similar to those given previously, and are therefore omitted.

\end{proof}

We first recover the well known result for the spin glass model.

\begin{corollary}
The Ising spin glass model given by
\[ h(\sigma) = - \sum_{ \{ i,j \} \in E(G) } {J_{i,j} \tau_i \tau_j}, \quad \text{and} \quad Z(G)=\sum_{\tau\in \mathcal{T}(G)} \e{-\beta h(\tau)} , 
 \]
has partition function 
\[ Z(G)=
\e{-\beta\sum_{ e \in E } J_{e}} Z_T (G; q=2, \{ \e{2\beta J_e}-1\}_{e\in E(G)}).
\]
\end{corollary}

More importantly, we have the following results for the RFIM.

\begin{corollary}
The RFIM  given by
\[ 
h(\tau) = - J\sum_{ \{ i,j \} \in E }  {\tau_i \tau_j} - \sum_{i \in V(G)}{z_i \tau_i},\quad \text{and} \quad Z(G)=\sum_{\tau\in \mathcal{T}(G)} \e{-\beta h(\tau)},
\]
has partition function 
\[  Z(G)= \e{-\beta(J|V(G)| +3 \sum_{i \in V(G)} z_i  )}  \V\left( G,\omega  ; \;\{x_{z}\}_{z\in \mathbb{C}}    ,\;    \e{2\beta J}-1        \right),\]
where the vertex weights $\omega$ are given by $\omega (v_i) =z_i$, 
 and, for any $z\in \mathbb{C}$, 
$  x_{z}=    \e{2z} + \e{4z} $.

Furthermore, $Z(G)$ is a polynomial in the variables $\left\{ \e{\beta J}-1,  x_{z} \; | \; z \in  \mathcal{M} \right\},$
where $\mathcal{M}=\left.\left\{   \sum_{i=1}^{|V(G)|}  \varepsilon_i z_i  \right|  \varepsilon_i = 0 \text{ or } 1  \right\} $.
 
$Z(G)$ also has Fortuin-Kasteleyn-type representation
\[Z(G)=\sum_{A\subseteq E(G)}  x_{z_{C_1}} \cdots x_{z_{C_{k(A)}}}  \left(\e{\beta J}-1 \right)^{|A|},\]
where $z_{C_l}$ is the sum of the weights, $z_i$, of all of the vertices $v_i$ in the $l$-th connected component of the spanning subgraph $(V(G), A)$. 
\end{corollary}

\section{Leveraging Prior Results}

\subsection{Computational complexity}

Realizing the Potts model partition function as an evaluation of the $\V$-polynomial now means that the computation complexity results for the $W$-polynomial apply directly to partition functions with an external field.  We collect some of these here, drawing on the work of Noble and Welsh \cite{NW99}, and Noble \cite{nob09}.

The computation complexity of the Tutte polynomial, and hence zero field Potts model partition function, has been extensively studied (see, in particular, Jerrum \cite{Jer87}, Jaeger, Vertigan, and Welsh \cite{JVW90} and also Vertigan \cite{Ver05}).  The conclusion of these investigations  is that computing the Tutte
polynomial is $\#P$-Complete for general graphs, except when $q=1$, or for a small set of special points, or for planar graphs when $q=2$. Furthermore, approximation is provably difficult as well: see \cite{ AFW94, AFW95, WM00} for overviews and \cite{GJ07, GJ07b, Jer07} for recent results in this area.

Thus, it is not surprising that the computational complexity consequences of Theorem~\ref{ZV} and Section~\ref{s.special} for the variable field Potts model are somewhat bleak.  Noble and Welsh \cite{NW99} have shown that computing any coefficient of the $W$-polynomial is 
$\sharp$P-hard even for trees, and  specific coefficients are $\sharp$P-hard for complete graphs.  Thus the complexity of the variable field Potts model is at least as problematic (presumably more so if variable interaction energies are also used).  Additionally, Noble and Welsh \cite{NW99} have shown that computing evaluations of the $W$-polynomial, and hence the Potts model partition function with external field, are $\sharp$P-hard not only for trees, but even just for stars. 

The prognosis in the case of a constant, but non-zero, magnetic field, is somewhat better.  As noted above, the $U$-polynomial corresponds to a constant magnetic field vector. Noble \cite{nob09} has shown that if $G$ is a graph with tree-width at most $K$, then  the $U$-polynomial, and hence the partition function  of $G$, may be evaluated at any point in roughly $O(a_Kn^{2K+3})$ arithmetic operations. 

Complexity results for the Ising model, essentially the $q=2$ Potts model, differ significantly from the general Potts model.  In particular, the partition function for the Ising model with zero field and constant interaction energies can be reformulated as a tractable problem for planar graphs (see \cite{Fis66, Kas67, JVW90}).  Considerable work has been done investigating the computational complexity of the variable field Ising model under under several conditions, notably by Goldberg and Jerrum \cite{GJ07b}, and Jerrum and Sinclair \cite{JS93}. These include complexity classifications and, where possible, approximation algorithms for the Ising model with different restrictions on the interaction energies and magnetic fields.  In particular there is no fully polynomial randomized approximation scheme (FPRAS) in the antiferromagetic case, but there is a FPRAS in the ferromagnetic case provided that all the local magnetic field values have the same sign. Without this restriction on the magnetic field values, the problem again becomes intractable.

However, Theorems \ref{h1} and \ref{Ising merge} imply that the Ising model with constant interaction energies and a constant magnetic field vector (without all entries necessarily having the same sign), is an evaluation of the $U$-polynomial.  Thus, by the results of Noble \cite{nob09}, it may be computed in polynomial time for graphs with bounded tree width.  It is very likely that this result may be improved by restricting to the $q=2$ case of the Ising model, and also likely that these results might be extended to variable interaction energies.   Applying the theory in the other direction, the complexity results for approximating the Ising model with external field from \emph{e.g.} \cite{GJ07b, JS93} now immediately apply to give computational complexity information for the $V$-polynomial when $q=2$.

\subsection{Zeros and phase transitions}

In addition to the computational complexity results discussed above, there has been considerable research into the zeros of the chromatic and multivariate Tutte polynomial.  From a graph theory perspective, this was traditionally motivated by graph colouring questions, since the chromatic polynomial of a graph $G$, when evaluated at a non-negative integer $q$, gives the number of ways to properly colour $G$ using $q$ colours.  From a statistical mechanics perspective, the interest stems from phase transitions which may be identified by (accumulation points of) zeros in the partition function. (See, for example,  \cite{BE-MPS10, Roy09, Shr01, Sok00, Sok05} and the references therein.)  For example, considerable effort has been  devoted to clearing regions of the complex plane, particularly those containing intervals of the real axis, of roots of the chromatic polynomial, thus precluding any phase transitions.  However, these investigations, particularly when theoretical as opposed to computer simulations, are largely in the absence of an external magnetic field.  The authors are not aware of any investigations of the zeros of the $W$-polynomial. It is likely that the techniques developed for the classical and multivariable Tutte and chromatic polynomials may be adapted to the $U$-, $\V$- and $W$-polynomials and thus extended to Potts models with external fields. 

\subsection{Knot theory connections}

Connections between knot theory and Potts model were first noted over twenty years ago by Jones \cite{Jo2}, and were made concrete by Kauffman \cite{Ka,Ka2} shortly thereafter. Such relations between statistical mechanics and knot theory were soon found for other quantum knot invariants by Jones  \cite{Jo} and  Turaev \cite{Tu}, and have since been explored  by many others. We refer the reader to either  Jones' paper \cite{Jo} or Wu's survey article \cite{Wu1} for an overview  of connections between statistical mechanical models and knot invariants. 
The relations between statistical mechanical models (particularly the Potts model and ice-type models) are of fundamental   and continued importance in knot theory. For example there has been recent interest in using topological versions of the Potts model to explore knot invariants  (see \cite{Mof} and the references therein).

Each of the relations between physics and knot theory mentioned above uses a graph based on the link diagram. The $U$-polynomial (and therefore the $W$- and $\V$-polynomials) can be used to construct knot invariants in a very different way. In fact, it is this different construction, due to Chmutov, Duzhin and Lando in the sequence of papers \cite{CDL94I, CDL94II, CDL94III}, that led Noble and Welsh to define the $U$- and $W$-polynomials. (The $U$- and $W$-polynomials arose through an investigation of the combinatorics  behind Chmutov, Duzhin and Lando's work on Vassiliev invariants.) This construction of knot invariants via the $U$-polynomial uses the theory of quantum and Vassiliev knot invariants (see \cite{Oht02}, for example, for an introduction to quantum and Vassiliev knot theory). In short, by applying the $U$-polynomial to the intersection graph of a chord diagram one obtains a weight system (a map $\mathcal{A}\rightarrow \mathbb{Q}$,where $\mathcal{A}$ is the algebra of chord diagrams  from the theory of Vassiliev invariants). One obtains a knot invariant by composing this weight system with the  Kontsevich invariant. The resulting knot invariants were categorized by Lieberum in \cite{Lie00}. By the results presented in this paper, these Vassiliev invariants provide a new connection between knot invariants and the Potts model.    

It is very provocative to have two very different connections between the Potts model and quantum knot invariants.  The problem of fully understanding the connections between these two applications of statistical mechanics to knot theory is one that the authors consider to be important.

\subsection{Other Hamiltonians}

While we have given examples of several Potts models that may be unified by the $\V$-polynomial, this list is not exhaustive, being only intended to illustrate the applications and techniques.  Because of the generality of the indexing set, even Theorem \ref{ZV2} might be adapted to other applications.  A more ambitious direction would be determining if non-linear terms in the Hamiltonian, such as the squared differences that appear in some biological models (see \cite{GG92, O+03} for example), might be assimilated into this theory in some way.

\subsection{Combinatorial properties}

Although the motivation for defining the $\V$-polynomial comes from statistical mechanics, it is an interesting combinatorial object in its own right, with potential applications in knot theory as well as statistical mechanics.  Thus, further study of its combinatorial properties may prove fruitful.  The papers \cite{AM09} and \cite{WF} are particularly relevant in this regard.  \cite{AM09} discovers a number of graph theoretical properties that are encoded by the Ising model with constant interaction energies and magnetic field.  Thus, all these properties are also encoded by the $\V$-polynomial.  A natural direction of investigation is to extend such results via the $\V$-polynomial to more general situations.  The $\Theta$ polynomial of \cite{WF} is especially interesting, and since it is equivalent to the RFIM, it is also a specialization of the $\V$-polynomial by Theorem \ref{Ising merge}.  This relation provides the appropriate transform of the $\Theta$ polynomial to relate it to the Tutte polynomial, and also gives a deletion-contraction reduction  for $\Theta$ that holds for non-constant vertex weights.  Furthermore, one of the results of \cite{WF} is that the $\Theta$ polynomial gives the ratio of the RFIM to the Bethe approximation, an important estimate of the partition function.  Thus, since both $\Theta$ and the RFIM are evaluations of the $\V$-polynomial, the Bethe approximation may be expressed as a ratio of $\V$-polynomials.  The connection between the $\V$-polynomial and $\Theta$ polynomial warrants further exploration.

\section*{Acknowledgements}
We thank Alain Brizard, Leslie Ann Goldberg, Maria Kiskowski, Steve Noble, Robert Shrock, Jim Stasheff, and especially Alan Sokal, for several informative conversations.


\bibliographystyle{elsarticle-num}

\end{document}